\documentclass[english, a4paper]{article}
\usepackage[T1]{fontenc}
\usepackage[cp1251]{inputenc}
\usepackage{amssymb}
\usepackage{babel}

\usepackage{amsfonts}

\usepackage{amsmath, amsthm, amscd, graphicx,amssymb}

\graphicspath{{figures/}}

\def\d{\rm d}
\def\th{\theta}
\def\tgamma{\tilde{\gamma}}
\def\tGamma{\tilde{\Gamma}}
\def\tx{\tilde{x}}
\def\ty{\tilde{y}}
\def\tth{\tilde{\th}}
\def\tp{\tilde{p}}
\def\tlambda{\tilde{\lambda}}
\def\tH{\tilde{H}}
\def\tilh{\tilde{h}}

\renewcommand{\i}{\mathrm{i}}
\newcommand{\Arg}{\operatorname{arg}}

\def\R{{\mathbb R}}
\def\Z{{\mathbb Z}}

\newcommand{\SE}{\operatorname{SE(2)}\nolimits}

\newcommand{\spann}{\operatorname{span}}
\newcommand{\tcusp}{t_{cusp}}

\newtheorem{theorem}{Theorem}

\newtheorem{remark}{Remark}

\begin{document}

\title{Relation between Euler's Elasticae and Sub-Riemannian Geodesics~on~$\SE$
\footnote{ The reported study was funded by RFBR, research project No. 16-31-60083 mol\_a\_dk}
}

\author{A.~Mashtakov, A.~Ardentov, Yu.~Sachkov \footnote{CPRC, Program Systems Institute of RAS}}

\maketitle

\begin{abstract}
In this note we describe a relation between Euler's elasticae and sub-Riemannian geodesics on $\SE$. Analyzing the Hamiltonian system of Pontryagin maximum principle we show that these two curves coincide only in the case when they are segments of a straight line.
\end{abstract}
\section*{Introduction}
In this paper we consider two classical geometric control problems~\cite{Jurdjevic, agrachev_sachkov}: the problem of Euler's elasticae and the problem of sub-Riemannian (SR) geodesics on $\SE$. Solution curves to both problems have many applications in mechanics~\cite{Euler, Sachkov2007, ElasticaAAA}, robotics~\cite{Laumond, SachkovSE2}, image processing~\cite{Mumford, Mirebeau, NMTMA, SIIMS} and modelling of human visual system~\cite{Citti_Sarti_B, DuitsJMIV2014}. Although solutions are well known in geometric control community, the authors have noticed a common confusion in applied societies where people sometimes mix these two curves. The reason is that the formulation of the problems is very similar, and at first sight one can wrongly deduce that Euler's elasticae are obtained via reparametrization of SR-geodesics by arclength. To prevent this possible confusion we clarify that these two curves coincide only in the case of straight line.

The structure of the paper is the following. First we briefly expose the history of the problems in Introduction. Afterwards, in Section~\ref{sec:statement}, we formulate both problems as optimal control problems on the Lie group $\SE$ and discuss two natural parameterizations for the solution curves. In Section~\ref{sec:PMP} we apply Pontryagin maximum principle for both problems and prove the main result in Theorem~\ref{thm1}. Finally, in Section~\ref{sec:Comparison}, we present several simulations with comparison of elasticae and SR-geodesics, which support our theoretical result.
\subsubsection*{History of problems}
In 1744 Leonhard Euler considered the problem on stationary configurations of an elastic rod with fixed endpoints and tangents at the endpoints~\cite{Euler}. Euler obtained differential equations for stationary configurations of the rod and described their possible qualitative types. These configurations are called Euler's elasticae.
Euler's elasticae are critical points of the elastic energy functional. The question as to which of the critical points are minima (local or global) was answered in~\cite{Sachkov2007, Sachkov_Elastica_1, Sachkov_Elastica_2, Sachkov_Elastica_3}. Afterwards the software to the boundary value problem (BVP) was proposed in~\cite{ElasticaAAA}, and thus a numerical realization of the optimal synthesis was obtained.

The sub-Riemannian problem on $\SE$ is formulated as follows. By given two unit vectors $v_0 = (\cos \th_0, \sin \th_0)$ and $v_1 = (\cos\th_1, \sin \th_1)$ attached respectively at two given points $q_0 = (x_0, y_0)$ and $q_1 = (x_1, y_1)$ in the plane, to find an optimal motion in the plane that transfers $q_0$ to $q_1$ such that the vector $v_0$ is transferred to the vector $v_1$. The vector can move forward or backward and rotate simultaneously. The required motion should be optimal in the sense of minimal length in the space $(x, y, \th)$, so called SR-length, where $\th$ is the slope of the moving vector. The problem can be seen as an optimal motion planning problem for the Reeds-Shepp car, which can move forward and backward and rotate on a place~\cite{Laumond}. It has important relations to vision~\cite{Citti_Sarti_B, DuitsJMIV2014} and image processing~\cite{NMTMA}. SR-geodesics are critical points of the SR-length functional. Explicit parametrization of the geodesics by Jacobi elliptic functions and elliptic integrals was obtained in~\cite{SachkovMoiseev}. A study of optimality of the geodesics was done in~\cite{SachkovSE2, Sachkov_SE2_2}, where the optimal synthesis was obtained. Here and below in the text by ''geodesic`` we mean a sub-Riemannian geodesic on $\SE$.

\section{Statement of problems}\label{sec:statement}
Both curves (elasticae and geodesics) are determined by optimal control problems on the Lie group $\SE$, the group of Euclidean motions of a plane.

The group $\SE$ is represented by $3 \times 3$ matrices~\cite{SachkovMoiseev}
$$\SE = \left \{\left(\begin{array}{c c c}
\cos\th & \sin\th & x\\
-\sin\th& \cos\th & y \\
0 & 0& 1 \\
\end{array} \right) \mid \th \in S^1, (x,y) \in \R^2\right \}.$$
The group can be naturally identified with the coupled space of positions and orientations $\R^2 \times S^1$ by identifying
$\left(\begin{array}{c c}
\cos\th & \sin\th\\
-\sin\th& \cos\th
\end{array} \right) \leftrightarrow \th\, (\!\!\!\mod 2\pi)$. Then for each $g = (x,y,\theta) \in \R^2 \times S^1 \cong \SE$ one has the left multiplication
$L_g g' = g g' = (x' \cos\th + y' \sin \th + x, -x' \sin\th + y' \cos \th + y, \th'+\th).$
Via the push-forward $(L_g)_*$ of the left-multiplication one gets the left-invariant frame $\{\mathcal{A}_1, \mathcal{A}_2, \mathcal{A}_3\}$ from the Lie-algebra basis $\{A_1,A_2,A_3\} = \{\left.\partial_x\right|_e, \left.\partial_\th\right|_e, \left.\partial_y\right|_e\}$ at the unity $e = (0, 0, 0)$:
%
\begin{equation*} \label{leftinvariant}
\mathcal{A}_{1}= \cos \theta \, \partial_{x} +\sin \theta \,
 \partial_{y}, 
 \quad
 \mathcal{A}_{2}= \partial_{\theta}, 
 \quad
 \mathcal{A}_{3}= -\sin \theta \, \partial_{x} +\cos \theta \,
 \partial_{y}. 
\end{equation*}

For a curve $\gamma(\cdot) = (x(\cdot),y(\cdot),\th(\cdot))$ on $\SE$ its projection $\Gamma(\cdot) = (x(\cdot),y(\cdot))$ on $\R^2$ is called the \emph{spatial projection}. In this note we consider only so-called horizontal curves, that satisfy the horizontality condition $\th(\cdot) = \Arg(\dot{x}(\cdot) + \i \, \dot{y}(\cdot)) \Leftrightarrow \dot{\gamma} \in \spann(\mathcal{A}_1,\mathcal{A}_2)$. There exist two natural parameterizations of horizontal curves on $\SE$ suitable for different needs (see~\cite{DuitsJMIV2014}):
\begin{itemize}
\item Sub-Riemannian (SR) arclength parameter $t$ is defined by $\|\dot{\gamma}(t)\| = 1$, where $\|\dot{\gamma}(t)\| = \sqrt{\xi^2\left(\left(\dot{x}\left(t\right)\right)^2 + \left(\dot{y}\left(t\right)\right)^2\right) + \left(\dot{\th}\left(t\right)\right)^2}$ for a given constant $\xi > 0$. Here and further in the text by dot we denote the derivative $\frac{d}{dt}$.
\item Spatial arclength parameter $s$ is defined by $\|\Gamma'(s)\| = 1$, where $\|\Gamma'(s)\| = \sqrt{\left(x'\left(s\right)\right)^2 + \left(y'\left(s\right)\right)^2}$. Here and further in the text we denote by prime the derivative $\frac{d}{ds}$. The spatial projection $\Gamma(\cdot)$ of a horizontal curve $\gamma(\cdot)$ can have singularities --- cusp points, which appear when the velocity vector $\dot{\gamma}(t)$ has nonzero component only in $\th$ direction.  The spatial arclength parameterization is well-defined only on segments of $\gamma(\cdot)$ whose spatial projection does not have cusps. For details see~\cite{DuitsJMIV2014}.
\end{itemize}

Euler's elastica problem is to find a $W_1^2[0,S]$ curve $\gamma:[0,S] \to \SE$, s.t.
\begin{equation*}\label{eq:elasticaorigsyst}
\begin{array}{c}
\gamma'(s) = \mathcal{A}_{1}|_{\gamma(s)} + u(s) \, \mathcal{A}_{2}|_{\gamma(s)}, \\[5pt]
\gamma(0) = g_0, \quad \gamma(S) = g_1,
\end{array}
\quad \int_0^{S} \frac{u(s)^2}{2} \,  \d s \to \min,
\end{equation*}
where the control $u: [0,S] \to \R$ is $L^2$ function, the initial point $g_0 = (x_0,y_0,\th_0)$ and the terminal point $g_1 = (x_1,y_1,\th_1)$ are given, and the terminal time $S>0$ is fixed and sufficiently large to guarantee that $g_0$ can be connected with $g_1$.

Notice that minimization of $ \int_0^{S} \frac{u^2}{2} \, ds$ in the elastica problem is equivalent to minimization of  $ \int_0^{S} \frac{\xi^2 + u^2}{2} \, ds$ for any $\xi>0$. Thanks to left-invariance we can set $g_0 = e = (0,0,0)$ without loss of generality.  Thus this problem is equivalent to
\begin{equation}\label{eq:elasticasyst}
\begin{array}{c}
\gamma'(s) = \mathcal{A}_{1}|_{\gamma(s)} + u(s) \, \mathcal{A}_{2}|_{\gamma(s)}, \\[5pt]
\gamma(0) = e, \quad \gamma(S) = g_1,
\end{array}
\quad \int_0^{S} \frac{\xi^2 + u(s)^2}{2} \, \d s \to \min.
\end{equation}

SR-problem on $\SE$ is to find a Lipschizian curve $\tgamma:[0,T] \to \SE$, s.t.
\begin{equation*}\label{eq:geodcontsystorig}
\begin{array}{c}
\dot \tgamma(t) = u_1(t) \, \mathcal{A}_{1}|_{\tgamma(t)} + u_2(t) \, \mathcal{A}_{2}|_{\tgamma(t)}, \\[5pt]
\tgamma(0) = g_0, \quad \tgamma(T) = g_1,
\end{array}
\quad \int_0^{T}\sqrt{\xi^2 u_1(t)^2 + u_2(t)^2} \, \d t \to \min,
\end{equation*}
where the controls $u_1, u_2: [0,T] \to \R$ are $L^\infty$ functions, the initial point $g_0$ and the terminal point $g_1$ are given, $\xi>0$ is a constant and the terminal time $T>0$ is free.

By Cauchy-Schwarz inequality and left-invariance the problem is equivalent to
\begin{equation}\label{eq:geodcontsystsquared}
\begin{array}{c}
\dot \tgamma(t) = u_1(t) \, \mathcal{A}_{1}|_{\tgamma(t)} + u_2(t) \, \mathcal{A}_{2}|_{\tgamma(t)}, \\[5pt]
\tgamma(0) = e, \quad \tgamma(T) = g_1,
\end{array}
\quad \int_0^{T}\frac{\xi^2 u_1(t)^2 + u_2(t)^2}{2} \, dt \to \min,
\end{equation}
where terminal time $T>0$ now is fixed.

\begin{remark}\label{remark:controls}
The control parameter $u(\cdot)$ represents the curvature of the spatial projection $\Gamma(\cdot) = (x(\cdot),y(\cdot))$ of a trajectory $\gamma(\cdot)$ in (\ref{eq:elasticasyst}), while the control $u_1(\cdot)$ and $u_2(\cdot)$ represent the ``spatial'' and ``angular'' components of velocity vector $\dot{\tgamma}(\cdot)$ in (\ref{eq:geodcontsystsquared}). Notice that $\Gamma(\cdot)$ is parameterized by spatial arclength $s$, i.e., $\|\Gamma'(s)\| = 1$ for $s \in [0,S]$, while the spatial projection $\tGamma(\cdot) = (\tx(\cdot),\ty(\cdot))$  of a trajectory $\tgamma(\cdot)$ in (\ref{eq:geodcontsystsquared}) satisfies $\|\dot \tGamma(t)\| = |u_1(t)|$ for $t \in [0,T]$.
\end{remark}

Comparing (\ref{eq:elasticasyst}) and (\ref{eq:geodcontsystsquared}) with account of Remark~\ref{remark:controls} one may wrongly deduce, that if the spatial projection $\tGamma(\cdot)$ is reparameterized by spatial arclength $s$ (corresponding to the control $u_1=1$) and the terminal time $S$ is chosen such that $S = \int_0^T u_1(t) \, dt$, then $\Gamma(\cdot)$ coincides with $\tGamma(\cdot)$. Although the substitution $t=s$, $T=S$, $u_1 = 1$ and $u_2 = u$ in (\ref{eq:geodcontsystsquared}) indeed gives (\ref{eq:elasticasyst}), we will show that $\Gamma(\cdot)$ coincides with $\tGamma(\cdot)$  only when $u_2 = u = 0$, i.e., when the both curves are segments of a straight line.
\section{Application of Pontryagin Maximum Principle}\label{sec:PMP}
 Here we apply Pontryagin Maximum Principle (PMP)~\cite{Pontryagin, agrachev_sachkov} to (\ref{eq:elasticasyst}) and (\ref{eq:geodcontsystsquared}). It can be shown, that abnormal extremals in (\ref{eq:elasticasyst}) are given by straight lines, and they are contained in the set of normal extremals (see~\cite{Sachkov2007}). In (\ref{eq:geodcontsystsquared}) abnormal extremals do not exist~\cite{SachkovMoiseev}. Thus we consider only the normal case.
\subsection{Hamiltonian System for Elasticae}\label{subsec:HamSystElast}
The control dependent Hamiltonian of PMP in problem (\ref{eq:elasticasyst}) reads as
\[\begin{array}{l}
H_u(\lambda,g) = \langle\lambda, \mathcal{A}_1 + u \, \mathcal{A}_2 \rangle -\frac{\xi^2 +  u^2}{2}, \text{ with } \lambda = \sum \limits_{k=1}^{3} p_k \,  {\rm d}  g^k \in T^{\ast}_g \SE,
\end{array}\]
where $\langle\cdot,\cdot\rangle$ denotes the action of a covector on a vector, and $({\rm d} g^1, {\rm d} g^2, {\rm d} g^3) = ({\rm d}  x, {\rm d}  y, {\rm d}  \theta)$ are basis one forms.

The maximization condition of PMP reads as
\begin{eqnarray*}
 \displaystyle{H_{u(s)}(\lambda(s),\gamma(s)) =
\max_{u\in\R}}\left(\langle\lambda(s), \mathcal{A}_1|_{\gamma(s)} + u \, \mathcal{A}_2|_{\gamma(s)} \rangle -\frac{\xi^2 +  u^2}{2}\right),
\end{eqnarray*}
where $u(s)$ denotes an extremal control and $(\lambda(s),\gamma(s))$ is an extremal.

The maximization condition gives the expression for the extremal control
\begin{equation*}
u(s) = \langle\lambda(s), \mathcal{A}_2|_{\gamma(s)}\rangle = p_3(s).
\end{equation*}
Then the maximized Hamiltonian reads as
\begin{equation*}\label{eq:theHamiltonianEl}
H = \langle\lambda, \mathcal{A}_1|_{\gamma}\rangle + \frac{\langle\lambda, \mathcal{A}_2|_{\gamma}\rangle^2- \xi^2}{2}= p_1 \cos\th + p_2 \sin \th + \frac{p_3^2 - \xi^2}{2}.
\end{equation*}
The Hamiltonian system with the Hamiltonian $H$ is defined as
\[ x' = \frac{\partial H}{\partial p_1}, \, \,
y' = \frac{\partial H}{\partial p_2}, \, \,
\th' = \frac{\partial H}{\partial p_3}, \qquad
p_1' = -\frac{\partial H}{\partial x}, \, \,
p_2' = -\frac{\partial H}{\partial y}, \, \,
p_3' = -\frac{\partial H}{\partial \th}.
\]
Thus, the Hamiltonian system for problem (\ref{eq:elasticasyst}) reads as
\begin{equation}
\label{eq:hamsys}
\begin{array}{ll}
\begin{cases}
p_1' = 0,\\
p_2' = 0, \\
p_3' = p_1 \sin\th - p_2 \cos \th
\end{cases} 
& \begin{cases}
x' =  \cos \th , \\
y' =  \sin \th, \\
\th' = p_3
\end{cases} \\[20pt] 
 \text{--- the vertical part,}
&  \text{--- the horizontal part,}
\end{array}
\end{equation}
with the boundary conditions
\begin{equation*}\label{eq:hamsysBound}
\begin{array}{l l l l l l}
p_1(0) = p_1^0, \,& \,p_2(0) = p_2^0, \, & \, p_3(0) = p_3^0, \, \, \, & \, \, \,
x(0) = 0, \,&\, y(0) = 0,  \,&\, \th(0) = 0.
\end{array}
\end{equation*}
\subsection{Hamiltonian System for SR Geodesics}\label{subsec:HamSystGeod}
The control dependent Hamiltonian of PMP in problem (\ref{eq:geodcontsystsquared}) reads as
\[\begin{array}{l}
\tH_u(\tlambda,g) = \langle\tlambda, u_1 \mathcal{A}_1 + u_2 \, \mathcal{A}_2 \rangle -\frac{\xi^2 u_1^2 +  u_2^2}{2}, \text{ with } \tlambda = \sum \limits_{k=1}^{3} \tp_k \,  {\rm d}  g^k \in T^{\ast}_g \SE.
\end{array}\]

The maximization condition of PMP reads as
\begin{eqnarray*}
 \displaystyle{\tH_{u(t)}(\tlambda(t),\tgamma(t)) =
\max_{(u_1,u_2)\in\R^2}}\left(\langle\tlambda(t), u_1 \, \mathcal{A}_1|_{\tgamma(t)} + u_2 \, \mathcal{A}_2|_{\tgamma(t)} \rangle -\frac{\xi^2 u_1^2 +  u_2 ^2}{2}\right).
\end{eqnarray*}

The maximization condition gives the expression for the extremal controls
\begin{equation*}
u_1(t) =  \frac{\tp_1(t) \cos\tth(t) + \tp_2(t) \sin\tth(t)}{\xi^2}, \qquad u_2(t) = \tp_3(t).
\end{equation*}
Then the maximized Hamiltonian reads as
\begin{equation*}\label{eq:theHamiltonianGeod}
\tH = \frac12\left(\frac{\left(\tp_1 \cos\tth + \tp_2 \sin \tth\right)^2}{\xi^2} + \tp_3^2\right).
\end{equation*}
The Hamiltonian system for problem (\ref{eq:geodcontsystsquared}) reads as
\begin{equation}
\label{eq:hamsysgeod}
\begin{array}{ll}
\begin{cases}
\dot{\tp}_1 = 0,\\
\dot{\tp}_2 = 0, \\
\dot{\tp}_3 = \frac{\left(\tp_1 \cos\tth + \tp_2 \sin\tth\right)\left(\tp_1 \sin\tth - \tp_2 \cos \tth\right)}{\xi^2}
\end{cases} 
& \begin{cases}
\dot{\tx} =  \frac{\tp_1 \cos\tth + \tp_2 \sin\tth}{\xi^2} \cos \tth , \\
\dot{\ty} =  \frac{\tp_1 \cos\tth + \tp_2 \sin\tth}{\xi^2} \sin \tth, \\
\dot{\tth} = \tp_3
\end{cases} \\[20pt] 
 \text{--- the vertical part,}
&  \text{--- the horizontal part,}
\end{array}
\end{equation}
with the boundary conditions
\begin{equation*}\label{eq:hamsysBoundGeod}
\begin{array}{l l l l l l}
\tp_1(0) = \tp_1^0, \,& \,\tp_2(0) = \tp_2^0, \, & \, \tp_3(0) = \tp_3^0, \, \, \, & \, \, \,
\tx(0) = 0, \,&\, \ty(0) = 0,  \,&\, \tth(0) = 0.
\end{array}
\end{equation*}
\begin{remark}\label{rm:hamsysgeodspar}
Switching to spatial arclength parameter $s(t) = \int_0^t u_1(\tau) {\rm d} \tau$ (well-defined before the first cusp) leads to the following Hamiltonian system:
\begin{equation}
\label{eq:hamsysgeodspar}
\begin{array}{ll}
\begin{cases}
\tp_1' = 0,\\
\tp_2' = 0, \\
\tp_3' = \tp_1 \sin\tth - \tp_2 \cos \tth
\end{cases} 
& \begin{cases}
\tx' =  \cos \tth , \\
\ty' =   \sin \tth, \\
\tth' = \frac{\tp_3 \,  \xi^2}{\tp_1 \cos\tth + \tp_2 \sin\tth}
\end{cases} \\[20pt] 
 \text{--- the vertical part,}
&  \text{--- the horizontal part.}
\end{array}
\end{equation}
\end{remark}
\subsection{Relation between Ealsticae and SR-Geodesics}\label{subsec:MainThm}
By analyzing (\ref{eq:hamsys}) and (\ref{eq:hamsysgeod}) one gets the following result.
\begin{theorem}\label{thm1}
Let $\Gamma = \{(x(s), y(s))\,|\, s \in [a,b]\}$ be an elastica, and let $\tgamma = \{(\tx(t),\ty(t),\tth(t))\,|\,t \in [\alpha, \beta]\}$ be a SR-geodesic such that $\dot{\tx}^2 + \dot{\ty}^2 \neq 0$ for all $t \in [\alpha, \beta]$. Further, let $\displaystyle s(t) = \int_\alpha^t \sqrt{\dot{x}^2 + \dot{y}^2} d\tau$, and let $t=t(s)$, $s\in[0,S]$, $S = s(\beta)$, be the inverse function.  If
\begin{equation}\label{eq:twocurves}
(x(s), y(s)) \equiv (\tx(t(s)), \ty(t(s))), \quad s \in (c,d)
\end{equation}
for some interval $(c,d) \subset [a,b] \cap [0,S]$, $c<d$, then the both curves $\Gamma$ and $\tGamma = \{(\tx(t), \ty(t))\,|\, t\in[\alpha,\beta]\}$ are straight line segments.
\end{theorem}
\begin{proof}
First, by scaling homothety (see~\cite{SachkovMoiseev}) we set $\xi=1$ without loss of generality. Now, we rewrite the Hamiltonian systems (\ref{eq:hamsys}) and (\ref{eq:hamsysgeod}) via the left-invariant Hamiltonians $h_i = \langle p, \mathcal{A}_i \rangle$ and $\tilh_i = \langle \tp, \mathcal{A}_i \rangle$:
\begin{equation}\label{eq:hamsyssmov}
\begin{array}{ll}
(\ref{eq:hamsys}) \Leftrightarrow \begin{cases}h_1' = h_3 h_2,\\ h_2' = -h_3,\\  h_3' = - h_1 h_2,\\ x' = \cos\th,\\ y' = \sin\th,\\ \th' = h_2,\end{cases}
\quad
&
\quad
(\ref{eq:hamsysgeod}) \Leftrightarrow \begin{cases}\dot{\tilh}_1 = \tilh_3 \tilh_2 ,\\ \dot{\tilh}_2 = -\tilh_3 \tilh_1,\\  \dot{\tilh}_3 = - \tilh_2 \tilh_1,\\ \dot{\tx} =  \tilh_1 \cos\tth,\\ \dot{\ty} =  \tilh_1 \sin\tth,\\ \dot{\tth} = \tilh_2,\end{cases}
\\[40pt]
 \text{~~~~--- the elastica system,}
\quad  & \quad
 \text{~~~~--- the geodesic system.}
\end{array}
\end{equation}
a) The maximally continued elastica $\{(x(s),y(s))\,|\, s \in \R\}$, $(x')^2+(y')^2 \equiv 1$, has curvature $\kappa(s) = \th'(s) = h_2(s)$. The elastica system has first integrals $h_1^2 + h_3^2 \equiv const$ (Casimir function) and $h_1 + \frac{h_2^2}{2} \equiv const$ (the Hamiltonian), which implies that $h_2(s)$ is bounded. Thus the function $\kappa(s)$ is bounded on each elastica: $|\kappa(s)| \leq M$, $s \in \R$.

b) For the maximally continued SR-geodesic $\{(\tx(t),\ty(t),\tth(t))\,|\, t \in \R\}$, $\dot{\tx}^2+\dot{\ty}^2+\dot{\tth}^2 \equiv 1$, define the first cusp time after $t = \beta$:
$$t_{cusp} = \inf\{t>\beta\,|\, (\dot{\tx}^2+\dot{\ty}^2)(t)=0\}.$$
The function $s(t) = \displaystyle \int_\alpha^t \sqrt{\dot{\tx}^2+\dot{\ty}^2} d\tau$ is real analytic and increasing for $t \in [\alpha, t_{cusp})$, thus one can define an inverse function $t = t(s)$, $s \in [0,s_{cusp})$,
$$s_{cusp} = \lim\limits_{t \to \tcusp-0} s(t) = \int\limits_\alpha^{t_{cusp}} \sqrt{\dot{\tx}^2+\dot{\ty}^2} d\tau.$$
So there exists a real analytic vector function $(\tx(t(s)),\ty(t(s)))$, $s \in [0, s_{cusp})$. Its image on $\R^2$ is a curve $\tGamma$, whose curvature is given by $\tilde{\kappa}(s) = \frac{\tilh_2(t(s))}{\tilh_1(t(s))}.$ 
Notice that at the cusp point $\tilh_1(t_{cusp}) = 0$ and $|\tilh_2(t_{cusp})| = 1$, thus $\lim\limits_{s \to s_{cusp}-0} \tilde{\kappa}(s) = \infty$.

%
%
Introducing a polar angle $\varphi \in [0, 4 \pi]$ in the plane $(\tilh_1, \tilh_2)$, from the geodesic system one gets (see~\cite{SachkovMoiseev} for details):
$$\dot{\tx}^2+\dot{\ty}^2 = \sin^2 \frac{\varphi}{2} = 0 \Leftrightarrow \varphi = 2 \pi n, \quad n \in \Z.$$
Let $(x(t),y(t),\th(t)) = Exp(\lambda, t)$, $\lambda \in C$.

If $\lambda \in C_4$ then $\dot{\tx}^2 + \dot{\ty}^2 \equiv 0$, thus identity~(\ref{eq:twocurves}) is impossible.

If $\lambda \in C_5$ then the curves $\Gamma$ and $\tGamma$ are straight line segments.

If $\lambda \in C_1 \cup C_2$ then
\begin{equation}\label{eq:tcuspC1C2}
t_{cusp} < +\infty \Rightarrow s_{cusp}<+\infty.
\end{equation} 

And if $\lambda \in C_3$, then either~(\ref{eq:tcuspC1C2}) or
$$t_{cusp}^{-} = \sup\{t<\alpha \, | \, (\dot{\tx}^2 + \dot{\ty}^2)(t)= 0\}> -\infty,$$
thus $s_{cusp}^{-} = \int\limits_{\beta}^{t_{cusp}^{-}} \sqrt{\dot{\tx}^2 + \dot{\ty}^2} d\tau > -\infty.$ 
Consequently, we can assume inequalities~(\ref{eq:tcuspC1C2}), possibly, after time reversal on $\tgamma$.  

c) Suppose we have
\begin{equation}\label{eq:twocurvest}
(x(s), y(s)) \equiv (\tx(t(s)), \ty(t(s))),
\end{equation}
for all $s \in [c,d] \subset [0,s_{cusp})$, $c<d$.

All the functions in~(\ref{eq:twocurvest}) are real analytic for $s \in (c,d)$ and are analytically continued to the interval $I=(0, s_{cusp})$. By the uniqueness theorem for analytic functions, identity~(\ref{eq:twocurvest}) holds for $s \in I$. Thus $\kappa(s) \equiv \tilde{\kappa}(s)$, $s \in I$. But $\kappa(s)$ is bounded on $I$, while $\lim\limits_{s \to s_{cusp}-0} \tilde{\kappa}(s) = \infty$, a contradiction.
\end{proof}

\section{Comparison of Elasticae and SR Geodesics}\label{sec:Comparison}
In this section we support our theoretical result by series of simulations, where we compare elasticae and SR-geodesics.

In the first simulation we fix the initial momentum $\lambda(0)$, integrate the Hamiltonian systems (\ref{eq:hamsys}) and (\ref{eq:hamsysgeodspar}) with this initial momentum, and plot the projection in the plane of the corresponding trajectories, see top row in Figure~\ref{fig:examp1}. The experiment shows that SR-geodesic provides a good local approximation for the elastica in a neighborhood of the origin, when the same initial momentum was used for both systems. Although, these two curves coincide one with another only in the case of straight line.

In the second simulation we show the difference between optimal elasticae and SR-minimizers (optimal geodesics) in solution of the boundary value problem (BVP). We organize the experiment as follows. Fix the initial point $g_0 = (0,0,0)$ and the terminal point $g_1$; compute the SR-minimizer $\tilde{\gamma}$ departing from $g_0$ and arriving at $g_1$; compute the length $l$ of the spatial projection $\tilde{\Gamma}$ of the minimizer $\tilde{\gamma}$; compute the optimal elastica $\gamma$ that connects $g_0$ with $g_1$ and has length $l$; plot the projection in the plane of both curves. See the bottom row in Figure~\ref{fig:examp1}. The experiment clearly shows the difference between these two curves, and again they coincide one with another only in the case of a segment of a straight line. Subsequently we show that the result is stable with respect to change of parameter $\xi$ that balance penalization of spatial and angular displacement of SR-geodesics, see the left column of the bottom row in Figure~\ref{fig:examp1}.

\begin{figure}[ht]
\centering
\begin{minipage}{0.31\hsize}
\includegraphics[width=\linewidth]{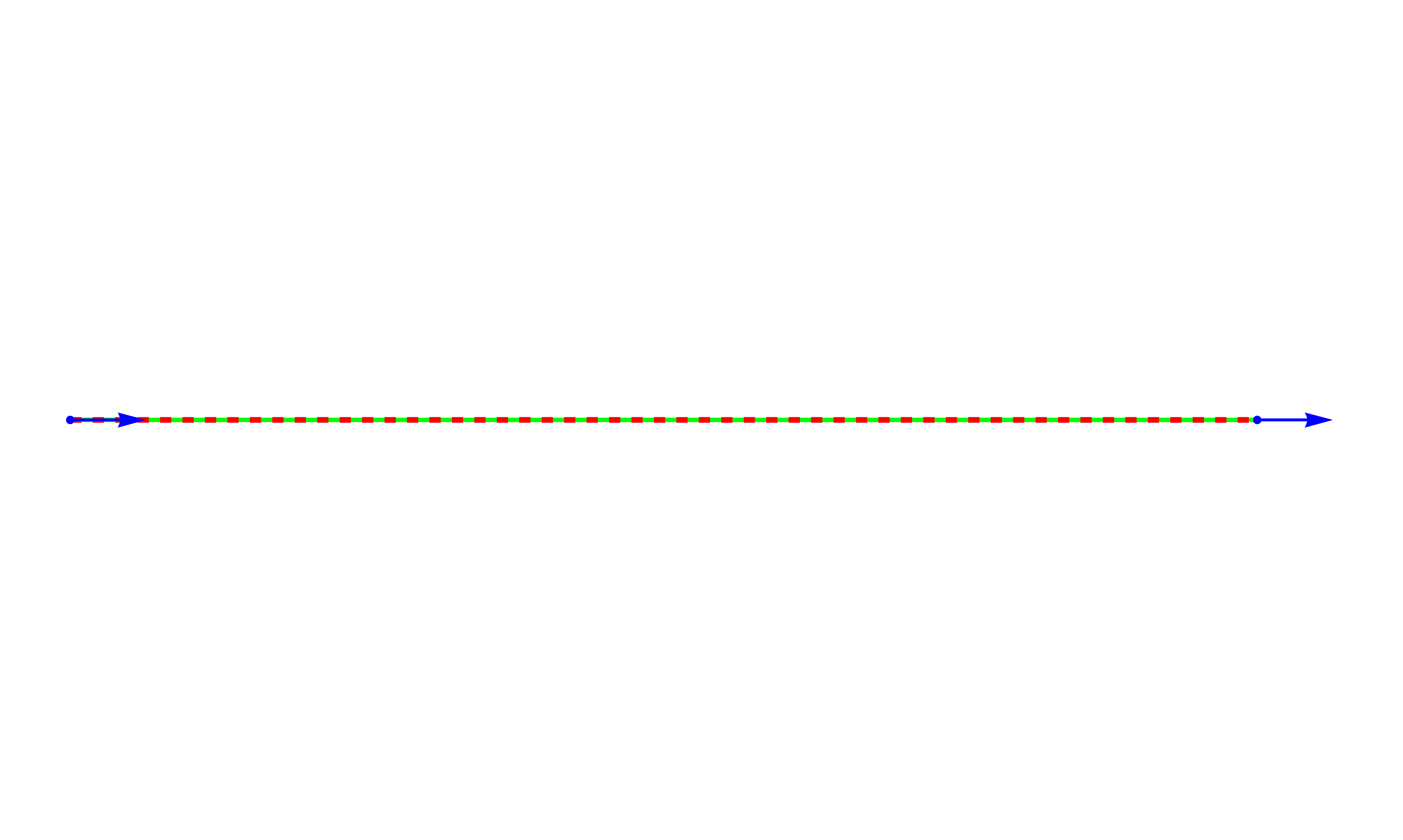}
\end{minipage}
~
\begin{minipage}{0.31\hsize}
\includegraphics[width=\linewidth]{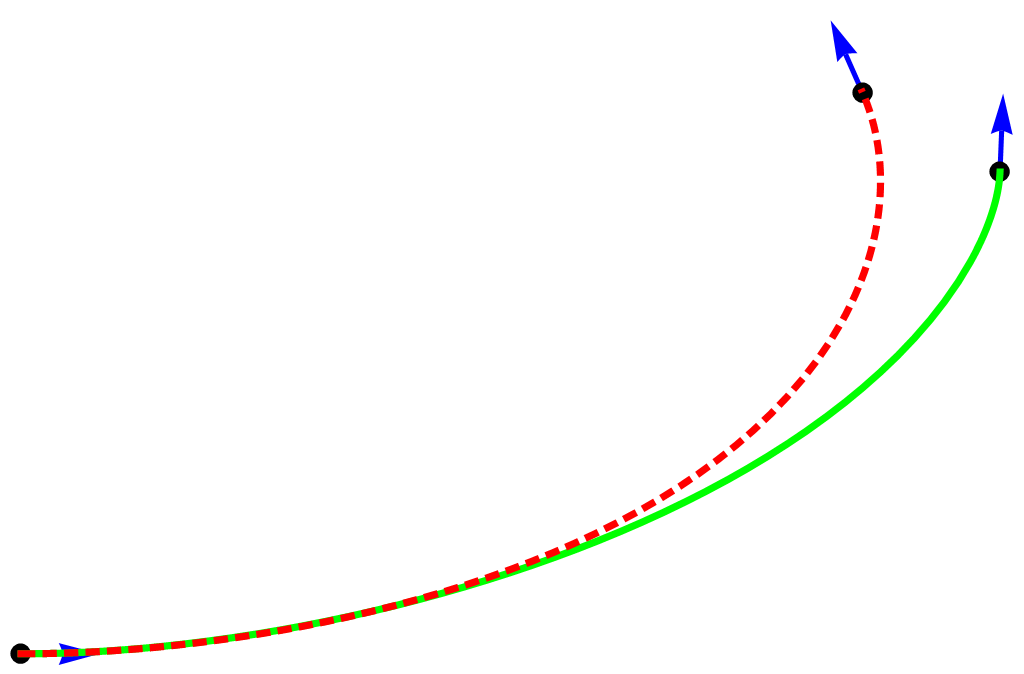}
\end{minipage}
~
\begin{minipage}{0.31\hsize}
\includegraphics[width=\linewidth]{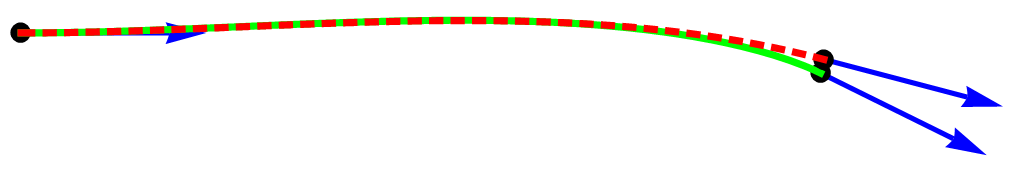}
\end{minipage}
\\
\begin{minipage}{0.31\hsize}
\includegraphics[width=\linewidth]{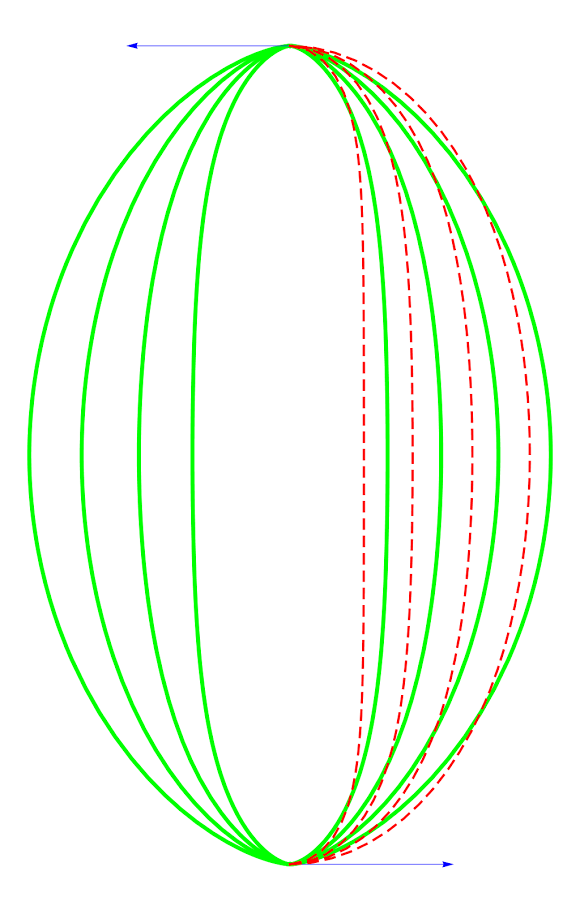}
\end{minipage}
~
\begin{minipage}{0.31\hsize}
\includegraphics[width=\linewidth]{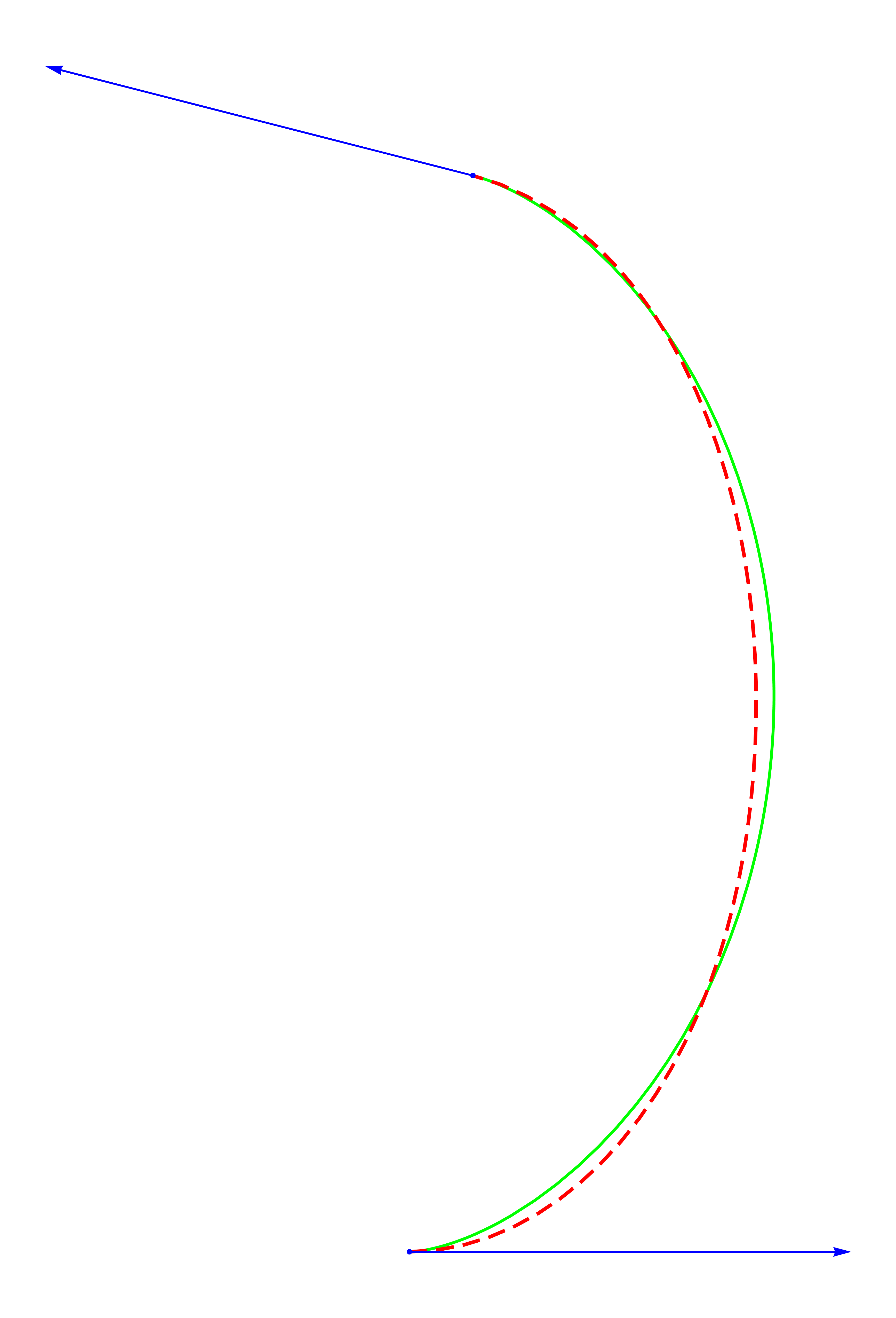}
\end{minipage}
~
\begin{minipage}{0.31\hsize}
\includegraphics[width=\linewidth]{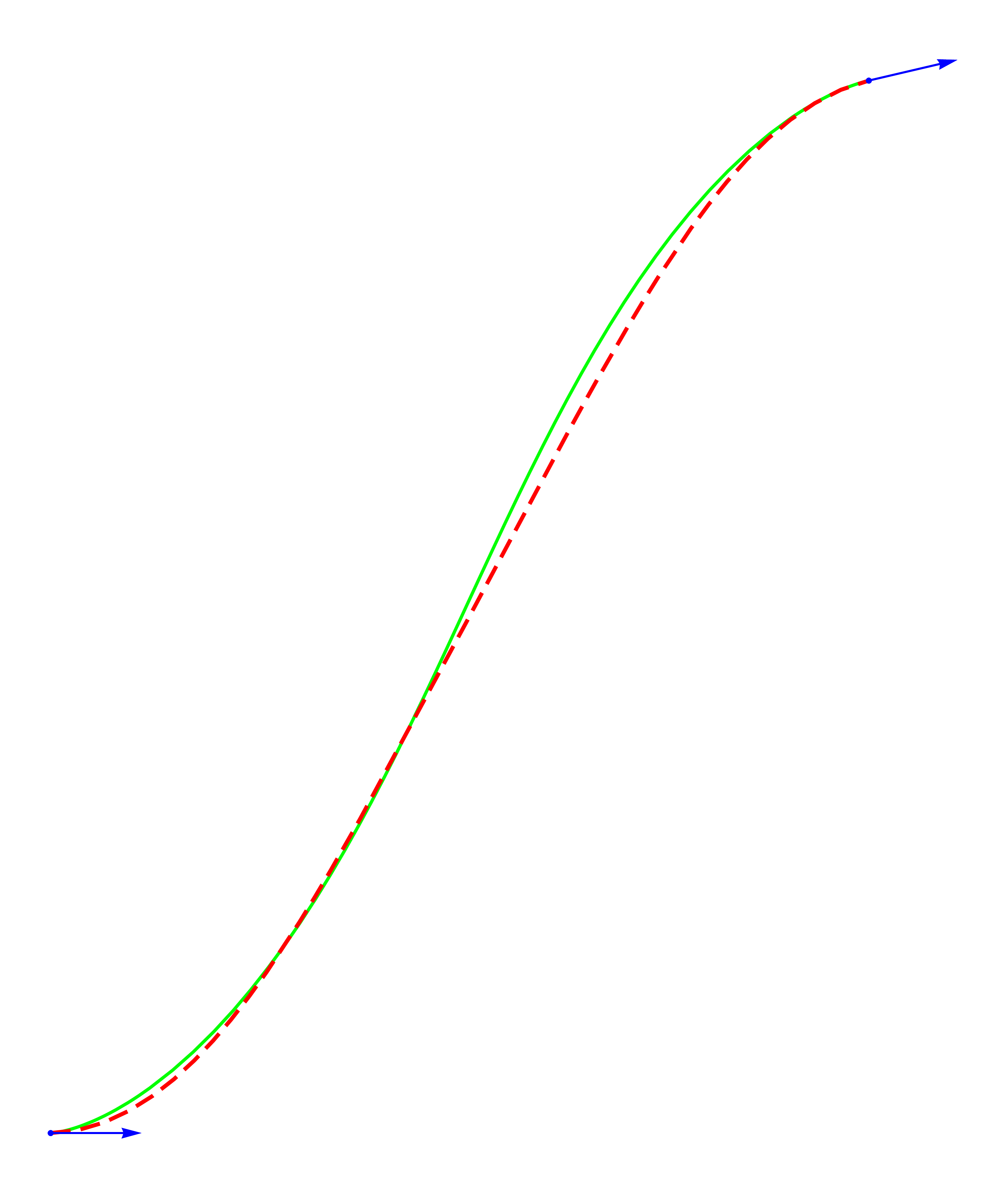}
\end{minipage}
\caption{Comparison of SR-geodesic (green solid line) and Euler's elastica (red dashed line).
\textbf{Top:} obtained by integration of (\ref{eq:hamsys}) and (\ref{eq:hamsysgeodspar}) with the same initial momenta $\lambda(0) = \{(0,0,1),(0.2,0.3, 0.95),(1.5, 0.35, 0.94)\}$ (from left to right).
\textbf{Bottom:} obtained as a solution to the boundary value problem with $g_0 = (0,0,0)$ and $g_1 = \{(0,1,-\pi),(0.03, 0.5, 2.9),(1.8, 2.3, 0.2)\}$ (from left to right), where the length of elastica is taken the same as the spatial length of SR-geodesic. In the left plot the parameter $\xi$ for SR-geodesics is varied. This corresponds to different length of elasticae.}
\label{fig:examp1}
\end{figure}

\end{document}